\RequirePackage{fix-cm}
\documentclass[smallextended]{svjour3}       
\smartqed  
\usepackage{graphicx}
\usepackage{color}
\usepackage{amsmath, amsfonts, amssymb}
\usepackage{bm}
\usepackage{hyperref}   
\usepackage{mathtools} 
\usepackage{algorithm}
\usepackage{algorithmic}
\usepackage{booktabs} 
\usepackage{comment}
\def\d{\mbox{d}}
 \journalname{Journal of Scientific Computing}
\begin{document}

\title{An adaptive radial basis function approach for efficiently solving multidimensional spatiotemporal integrodifferential equations
}

\titlerunning{Adaptive RBF approach for spatiotemporal equations}        

\author{Mingtao Xia      \and
        Qijing Shen 
}


\institute{Mingtao Xia (corresponding author) \at
              Department of Mathematics, University of Houston, Houston, Texas, 77204 \\
              School of Mathematics, University of Birmingham, Edgbaston, Birmingham, B15 2TT\\
              \email{mxia4@uh.edu}           
           \and
           Qijing Shen \at
              Nuffield School of Medicine, University of Oxford, Oxford, OX3 7BN, United Kingdom\\
              \email{qijing.shen@reuben.ox.ac.uk}
}

\date{Received: date / Accepted: date}

\maketitle

\begin{abstract}
In this work, we propose an adaptive radial basis function (RBF) approach for the efficient solution of multidimensional spatiotemporal integrodifferential equations. Our approach can automatically adjust the shape of RBFs and provide an easy-to-implement and mesh-free approach for solving spatiotemporal equations. 
Specifically, we analyze how the proposed method mitigates the curse of dimensionality by adaptively adjusting the scales and centers of the radial basis functions when the solution is spatially anisotropic. Through a range of numerical examples, we demonstrate the effectiveness of our approach for solving multidimensional spatiotemporal integrodifferential equations.
\keywords{Radial Basis Function \and Numerical PDE \and Multidimensional Spatiotemporal Integrodifferential Equation \and Neural ODE}
\subclass{	65M15  \and 65M70 \and 65N99}
\end{abstract}

\section{Introduction}

Multidimensional spatiotemporal equations, commonly formulated as partial differential equations (PDEs) or integrodifferential equations, play a fundamental role in modeling complex systems that evolve across both space and time. Such equations arise in a wide range of physical, biological, and engineering applications, including reaction–diffusion processes \cite{murray2002mathematical}, fluid dynamics \cite{batchelor2000introduction}, and population dynamics \cite{perthame2007transport}. In higher-dimensional settings, these models frequently incorporate nonlocal interactions to describe phenomena such as turbulent transport \cite{pope2000turbulent}, the propagation of neural activity in the brain \cite{ermentrout1998neural}, and nonlocal material responses in elasticity and fracture mechanics \cite{silling2000reformulation}.

Despite their broad applicability, analytical solutions to multidimensional spatiotemporal integrodifferential equations are rarely available due to their intrinsic complexity, strong nonlinearities, and the presence of nonlocal terms. As a result, numerical methods play a crucial role in analyzing and simulating such systems. Classical approaches include finite difference, finite element, and spectral methods, each offering distinct advantages in terms of simplicity, flexibility, or accuracy \cite{trefethen2000spectral,johnson2012finite}. Finite difference methods are straightforward and well-suited to structured grids, finite element methods accommodate complex geometries and boundary conditions \cite{zienkiewicz2005finite}, and spectral methods can achieve high accuracy with relatively few degrees of freedom for smooth solutions \cite{boyd2001chebyshev}. However, as the spatial dimension increases, these traditional techniques often suffer from the curse of dimensionality, whereby computational costs—such as the number of grid points or basis functions—grow exponentially with the dimension. This rapid increase in complexity severely limits the feasibility of high-dimensional simulations, motivating the development of more efficient numerical methodologies \cite{hackbusch2012tensor}.

To address this challenge, various efficient numerical approaches tailored for multidimensional tasks, such as sparse grids, low-rank tensor methods, and adaptive multiscale schemes, have been developed, providing promising strategies for efficiently solving some high-dimensional spatiotemporal problems \cite{hutzenthaler2020overcoming,bungartz2004sparse,bachmayr2023low}. Yet, most existing approaches are for static problems, \textit{e.g.}, heat equations. For spatiotemporal equations, the behavior of the solutions to them may evolve temporally \cite{chou2023adaptive,xia2021efficient}, requiring adaptive adjustment of the numerical scheme. An adaptive hyperbolic-cross-space spectral method was recently developed to efficiently solve certain multidimensional spatiotemporal integrodifferential equations \cite{deng2025adaptive,xia2024learning}. However, there are several restrictions when adopting this adaptive spectral method. First, constraints are often imposed on the spatial domain-for example, requiring it to be a box or $\mathbb{R}^d$-so that a set of orthogonal basis functions can be constructed explicitly. Second, the solution of the spatiotemporal equation is typically assumed to be sufficiently smooth to guarantee exponential convergence. Third, selecting an optimal hyperbolic cross space generally requires prior knowledge of the solution \cite{shen2010sparse,luo2013hermite}, which may be unavailable in practice.

Radial basis function (RBF) methods have become a powerful and flexible tool for solving multidimensional PDEs on irregularly shaped domains. Unlike traditional mesh-based methods such as finite difference or finite element methods, RBF approaches are meshfree and dimension-independent, making them especially suitable for problems defined on complex geometries or evolving domains \cite{fasshauer2007meshfree}. In these methods, the solution of a PDE is approximated by a linear combination of radial basis functions centered at data points, allowing for high-order accuracy and spectral convergence when infinitely smooth RBFs, such as Gaussian or multiquadric functions, are used \cite{flyer2009radial}. Due to these advantages, RBF-based numerical schemes have been successfully applied to numerically solving a wide range of multidimensional PDEs. When adopting RBF-based numerical approaches, proper scales and centers for RBFs are essential for reducing approximation error and improving accuracy \cite{billings2007generalized,chen2011numerical,zhang2017adaptive}. In particular, when solving spatiotemporal integrodifferential equations, the solution behavior may evolve over time, necessitating appropriate adaptation of the radial basis functions. Previous work has shown that adjusting the RBF centers while keeping the scales fixed can be essential for numerical success when solving one-dimensional time-dependent PDEs \cite{sarra2005adaptive}. However, to the best of our knowledge, there has been limited research on the development of efficient methods that can adaptively and automatically evolve both the centers and scales of RBFs for the solution of multidimensional spatiotemporal integrodifferential equations.

In this manuscript, we propose an adaptive RBF approach for the efficient solution of multidimensional spatiotemporal integrodifferential equations. By integrating the neural ordinary differential equation (neural ODE) framework \cite{chen2018neural}, the proposed method automatically adjusts the centers and scales of the RBFs over time in response to the evolving solution. Unlike traditional mesh-based methods such as finite difference and finite element methods, the proposed approach is mesh-free and therefore straightforward to implement. In contrast to spectral methods, the adaptive RBF framework can be applied to general spatial domains-including bounded, unbounded, and irregularly shaped domains-and does not require the solution to be highly smooth. Moreover, because the RBFs are global in space, the proposed method provides a unified framework for solving both spatiotemporal PDEs and more general integrodifferential equations. This distinguishes our approach from many existing machine-learning-based methods, such as physics-informed neural networks \cite{raissi2019physics,cuomo2022scientific} and PDE-Net \cite{long2018pde}, which primarily target PDEs and are less naturally suited to integrodifferential formulations.
We shall demonstrate the effectiveness of our adaptive RBF approach for solving multidimensional spatiotemporal integrodifferential equations from both theoretical and numerical aspects, and the main contributions of our work are as follows:
\begin{enumerate}
\item We integrate the neural ODE framework with RBFs and develop an adaptive RBF method. Our approach is mesh-free and simple to implement, providing a unified and efficient way to solve multidimensional spatiotemporal integrodifferential equations.
\item We conduct an analysis for RBF-based approximations of spatiotemporal integrodifferential equations and derive an explicit error bound. This analysis shows that the proposed adaptive RBF strategy can partially alleviate the ``curse of dimensionality" when the solution exhibits spatial anisotropy.
\item We demonstrate the effectiveness of the adaptive RBF approach through a variety of numerical experiments on spatiotemporal equations, highlighting that adaptive adjustment of the RBF shapes and centers is essential for accurate numerical performance.
\end{enumerate}

The remainder of this manuscript is organized as follows. In Section~\ref{section2}, we present and analyze the proposed adaptive RBF approach for solving multidimensional spatiotemporal integrodifferential equations. In Section~\ref{section3}, we demonstrate the effectiveness of the method through a series of numerical examples. Finally, in Section~\ref{section4}, we summarize the main findings and discuss several directions for future research.

\section{Adaptive RBF approach for solving multidimensional spatiotemporal integrodifferential equations}
\label{section2}
In this section, we present and analyze the adaptive strategies used to apply radial basis functions for solving multidimensional spatiotemporal equations. For simplicity, in our theoretical analysis, we focus on the following general form of a spatiotemporal equation subject to homogeneous Dirichlet boundary conditions:
\begin{equation}
\begin{aligned}
    &\frac{\partial u(\bm{x}, t)}{\partial t} = A(u(\bm{x}, t), \bm{x}, t), \bm{x}, t), \,\,\bm{x}\in\Omega, t\in[0, T],\\
    & u(\bm{x}, 0) = u_0(\bm{x}), \,\, u(\bm{x}, t) = 0, \,\,\bm{x}\in\partial\Omega.
    \end{aligned}
    \label{model_problem}
\end{equation}
In Eq.~\eqref{model_problem}, $A$ denotes a general integrodifferential operator acting only on the spatial variable $\bm{x}$. We assume that the model problem \eqref{model_problem} is well posed and admits a unique solution $u(\bm{x},t)\in C^1([0,T],F_k^d)$ for some $k\ge 1$, where $F_k^d$ denotes the function space introduced in \cite{barthelmann2000high}:
\begin{equation}
\begin{aligned}
&F_k^d \coloneqq  \Big\{ f \coloneqq [-1, 1]^d \to \mathbb{R}  \big|\ D^{\bm{\alpha}}f \text{ continuous if } \alpha_i \leq k \text{ for all } i,\\
&\hspace{5cm}\bm{\alpha}\coloneqq(\alpha_1,...,\alpha_d)\in \mathbb{N}_0^d \Big\}
\end{aligned}
\end{equation}
equipped with the norm:
\begin{equation}
\|f\|_{k, \infty} = \max\left\{ \left\|D^{\alpha}u\right\|_{\infty} \ \middle|\ \alpha \in \mathbb{N}_0^d,\ \alpha_i \leq k \right\}.
\end{equation}

Our adaptive RBF approach approximates the solution $u(\bm{x}, t)$ to Eq.~\eqref{model_problem} using spatial radial basis functions with time-varying coefficients, scales, and centers:
\begin{equation}
    u(\bm{x}, t)\approx u_n(\bm{x}, t)= \sum_{i=1}^N c_i(t)B\big(\bm{\epsilon}_i(t)*(\bm{x}-\bm{x}_i(t))\big), \,\, \bm{x},\bm{x}_i\in\Omega\subseteq \mathbb{R}^d,\,\, t\in\mathbb{R}^+.
    \label{spatiotemporal_approx}
\end{equation}
In Eq.~\eqref{spatiotemporal_approx}, $*$ denotes the element-wise Hadamard product. The vectors $\bm{\epsilon}_i(t)\coloneqq (\epsilon_{i,1}(t),\ldots,\epsilon_{i,d}(t))$ and $\bm{x}_i(t)\coloneqq (x_{i,1}(t),\ldots,x_{i,d}(t))$ represent the multidimensional, heterogeneous scales and centers (displacements) of the RBFs, respectively. Throughout this manuscript, we adopt the multivariate Gaussian kernel as the RBF:
\begin{equation}
    B\big(\bm{\epsilon}_i(t)*(\bm{x}-\bm{x}_i(t))\big)\coloneqq B_{\bm{\epsilon}_i, \bm{x}_i}(\bm{x})=\prod_{i=1}^d\Big(\frac{1}{\sqrt{\pi\epsilon_i^{-2}}}\cdot \exp\big(-\epsilon_i^2{x}_i^2\big)\Big). 
\end{equation}
In this manuscript, $\|\cdot\|$ refers to the $l^2$ norm of a vector or the $L^2$ norm of a function. 

\subsection{Approximation error of multivariate RBF in space}
In this subsection, we analyze the approximation error associated with the spatial radial basis function approximation \eqref{spatiotemporal_approx} for a fixed time $t$. Drawing on ideas from sparse grid methods and hyperbolic cross spaces in spectral approximation theory \cite{shen2010sparse,barthelmann2000high}, we show that the use of spatial RBF approximations can partially alleviate the curse of dimensionality when the solution exhibits spatial anisotropy. For simplicity, we restrict our attention to the case where the spatial domain in \eqref{spatiotemporal_approx} is given by $\Omega=[-1, 1]^d$ in Eq.~\eqref{spatiotemporal_approx} for our theoretical analysis. Under this setting, we establish the following result.

\begin{theorem}
\label{theorem1}
\rm
Let $\bm{\epsilon}^{-1} \coloneqq (\epsilon_1^{-1}, \ldots, \epsilon_d^{-1})$ with $\epsilon_i < 1$ for $i = 1, \ldots, d$. 
For a function $u(\bm{x}) \in F_k^d$, $\bm{x} \in \Omega$, we assume that $u(\bm{x})$ vanishes on the boundary $\partial \Omega$. 
Under this assumption, there exists a continuously differentiable extension $\tilde{u}$ of $u$ to $\mathbb{R}^d$ such that
$u(\bm{x}) \equiv 0$ for $\bm{x} \notin [-1,1]^d$, and
$\tilde{u}(\bm{x}) = u(\bm{x}), \forall\bm{x} \in \Omega$.
Then, given $N$ radial basis functions, there exists an RBF approximation of the form:
\begin{equation}
    u_N(\bm{x})=\sum_{i=1}^N c_iB_{\bm{\epsilon}_i^{-1}, \bm{x}_i}(\bm{x})
\end{equation}
satisfying:
\begin{equation}
\begin{aligned}
    \|u(\bm{x})-u_N(\bm{x})\|&\leq \bigg(\sum_{i=1}^d\epsilon_i^{\frac{1}{2}}\|\partial_{x_i}u\|_{\infty, 0} + 2\|u\|_{\infty, 0}\big(1-\Phi(\bm\epsilon^{-\frac{1}{2}})\big)\bigg) \\
    &\quad+ \prod_{i=1}^d\epsilon_i^{-k}2^{d+2k}c_{d, k}N^{-\frac{k}{2}}(\log N)^{(k+2)(d-1)+1}\|u\|_{\infty, k}\|B_{\bm{1}, \bm{0}}\|_{\infty, k},
    \end{aligned}
    \label{error_bound}
\end{equation}
where $c_{d, k}$ is a constant depending on the dimensionality $d$ and $k$, $\bm{1}=(1,...,1)\in\mathbb{R}^d, \bm{0}=(0,...,0)\in\mathbb{R}^d$, and 
\begin{equation}
    \Phi(\bm\epsilon^{-\frac{1}{2}})\coloneqq \prod_{i=1}^d\int_{|{x}_i|\leq\epsilon_i^{-\frac{1}{2}}}\frac{1}{\sqrt{2\pi}}\exp(-\frac{x_i^2}{2})\d x_i.
\end{equation}
\end{theorem}

The proof of Theorem~\ref{theorem1} is provided in Appendix~\ref{proof_theorem1}. From Eq.~\eqref{error_bound}, if we choose homogeneous scales $\epsilon_i \coloneqq N^{-1/(2d)}$ for all $i$, then the right-hand side is of order $\mathcal{O}\!\left(N^{-1/(4d)}\right)$, which reflects the curse of dimensionality. However, the spatial RBF approximation can partially alleviate this issue when the solution $u$ is spatially anisotropic, in the sense that the quantities $\|\partial_{x_i} u\|_{\infty,0}$ vary significantly across different dimensions. 
Such spatially anisotropic functions arise in a wide range of applications. For example, in hydrogeology, they describe groundwater flow that moves readily along permeable layers but much more slowly across clay barriers \cite{gelhar1986stochastic,rubin2003applied}. In seismology, they are used to characterize wave speeds that vary with propagation direction in anisotropic rock formations \cite{thomsen1986weak}.

To exemplify, we consider two cases. First, we assume that there exists a $c>0$ such that $\|\partial_{x_i}u\|_{\infty, 0}\leq \exp(-ci)$. In this case, we can choose $\epsilon_i = N^{-\frac{1}{4i(i+1)}}$. Since the curmulative distribution function $\Phi_1(x)$ of the standard normal distribution $\mathcal{N}(0, 1)$ satisfies $1-\Phi_1(x)<\frac{\exp(-\frac{x^2}{2})}{\sqrt{2\pi}x}$ for $x>0$, we have:
\begin{equation}
    1 - \Phi(\bm\epsilon^{-\frac{1}{2}})\leq \sum_{i=1}^d\frac{2}{\sqrt{2\pi\epsilon_i}}\exp(-\frac{\epsilon_i^{-1}}{2})\leq \frac{2dN^{\frac{1}{8d(d+1)}}}{\sqrt{2\pi}}\exp(-\frac{N^{\frac{1}{4d(d+1)}}}{2}).
\end{equation}
On the other hand, $\prod_{i=1}^d\epsilon_i^{-k}N^{-\frac{k}{2}}\leq N^{-\frac{k}{4}}$. Furthermore, we assume that $N\leq \exp(16c)$. Then, $\epsilon_i^{\frac{1}{2}}\exp(-ci)$ is decreasing in $i$. Therefore, the RHS of Eq.~\eqref{error_bound} can be simplified to:
\begin{equation}
\begin{aligned}
&dN^{-\frac{1}{16}} + 2\|u\|_{\infty, 0}\frac{2dN^{\frac{1}{8d(d+1)}}}{\sqrt{2\pi}}\exp(-\frac{N^{\frac{1}{4d(d+1)}}}{2}) \\
&\quad\quad+ 2^{d+2k}c_{d, k}N^{-\frac{k}{4}}(\log N)^{(k+2)(d-1)+1}\|u\|_{\infty, k}\|B_{\bm 1, \bm0}\|_{\infty, k},
\end{aligned}
\end{equation}
which converges faster to 0 than the usual $\mathcal{O}\!(N^{-\frac{k}{d}})$ as the dimensionality $d$ increases. 

As another case, we consider the scenario in which there exist $d_0 \ll d$ dominant directions in the sense of the spatial derivatives $\partial_{x_i} u$. Specifically, we assume that $\|\partial_{x_i} u\|_{\infty,0} < 1$ for $\quad i = 1,\ldots,d_0$ and $\|\partial_{x_i} u\|_{\infty,0} < \epsilon_0$ for $i = d_0+1,\ldots,d$ with $\epsilon_0 \ll 1$ satisfying $\epsilon_0^{-\frac{d_0}{d}}\exp(-\frac{\epsilon_0^{-\frac{2d_0}{d}}}{2})<\epsilon_0$.
We can set:
\begin{equation}
    \epsilon_i=\epsilon_0^2\cdot\epsilon_0^{\frac{2d_0}{d}}, i\leq d_0,\,\, \epsilon_i=\epsilon_0^{\frac{2d_0}{d}}, i>d_0.
    \label{case2_1}
\end{equation}
Then, we have:
\begin{equation}
    1 - \Phi(\bm\epsilon^{-\frac{1}{2}})\leq \sum_{i=1}^d\frac{2}{\sqrt{2\pi\epsilon_i}}\exp(-\frac{\epsilon_i^{-1}}{2})\leq \frac{2d\epsilon_0^{-\frac{d_0}{d}}}{\sqrt{2\pi}}\exp(-\frac{\epsilon_0^{-\frac{2d_0}{d}}}{2}).
    \label{case2_2}
\end{equation}
Plugging Eqs.~\eqref{case2_1} and \eqref{case2_2} into Eq.~\eqref{error_bound}, the RHS of Eq.~\eqref{error_bound} can be bounded by:
\begin{equation}
\begin{aligned}
        &d\epsilon_0^{\frac{d_0}{d}+1} + 2\|u\|_{\infty, 0}\frac{2d\epsilon_0^{-\frac{d_0}{d}}}{\sqrt{2\pi}}\exp\big(-\frac{\epsilon_0^{-\frac{2d_0}{d}}}{2}\big)\\
        &\quad+\epsilon_0^{-4d_0k}2^{d+2k}c_{d, k}N^{-\frac{k}{2}}(\log N)^{(k+2)(d-1)+1}\|u\|_{\infty, k}\|B_{\bm 1, \bm0}\|_{\infty, k}\\
    &\quad\quad\leq d\epsilon_0 + 2\|u\|_{\infty, 0}\frac{2d\epsilon_0}{\sqrt{2\pi}} \\
    &\hspace{1.5cm}+ \epsilon_0^{-4d_0k}2^{d+2k}c_{d, k}N^{-\frac{k}{2}}(\log N)^{(k+2)(d-1)+1}\|u\|_{\infty, k}\|B_{\bm 1, \bm0}\|_{\infty, k}.
\end{aligned}
    \label{error_bound1}
\end{equation}

Eq.\eqref{error_bound1} indicates that, to achieve an accuracy on the order of $\epsilon_0$-corresponding to the magnitude of the norms of the first-order partial derivatives in the insignificant directions, $\|\partial_{x_i} u\|_{\infty,0}$ for $i > d_0$-it suffices to use $N \sim \epsilon_0^{-8 d_0}$ basis functions, rather than the conventional requirement $N \sim \epsilon_0^{-d}$.


\subsection{Adaptive strategies for dynamically adjusting RBFs}
In this subsection, we present adaptive strategies for dynamically adjusting the centers and scales of the radial basis functions when solving spatiotemporal integrodifferential equations. Given the RBF approximation Eq.~\eqref{spatiotemporal_approx}, it satisfies the following spatiotemporal equation:
\begin{equation}
\begin{aligned}
    &\frac{\partial u_N}{\partial t}= \hat{A}(u_N, \bm{x}, s) \coloneqq \sum_{i=1}^N \partial_tc_i(t)\cdot B_{\bm{\epsilon}(t), \bm{x}_i(t)}(\bm{x}) \\
    &\hspace{2cm}+ \sum_{i=1}^N c_i(t)\cdot \Big(\sum_{j=1}^d\partial_j B_{\bm{\epsilon}_i(t), \bm{x}_i(t)}(\bm{x})\cdot \big(\partial_t\epsilon_{i, j}(t)\cdot (x_j-x_{i, j}(t)) \\
    &\hspace{7cm}- \partial_tx_{i, j}(t)(x_j-x_{i, j}(t))\big)\Big).
    \end{aligned}
    \label{approximate_model}
\end{equation}

We can prove the following result.
\begin{theorem}
\label{theorem2}
\rm
Assume $A$ in Eq.~\eqref{model_problem} is a linear operator and satisfies the following coercivity and linearity conditions: 
\begin{equation}
\begin{aligned}
    \int_{\Omega} u(\bm{x}, t)A(u(\bm{x}, t), \bm{x}, t)\d\bm{x}\leq  L\|u\|^2,\,\, A(u+v, \bm{x}, t) = A(u, \bm{x}, t)+ A(v, \bm{x}, t),
    \end{aligned}
    \label{A_condition}
\end{equation}
where $G$ is a function on $u$ as well as its derivatives on the boundary $\partial\Omega$.  
Then, the following error bound holds:
\begin{equation}
\begin{aligned}
    \big\|u(\bm{x}, t)-u_N(\bm{x}, t)\big\|^2&\leq \exp\big((2L+1)t\big)\cdot \bigg(\int_0^t \|A(u_N,\bm{x}, s)-\hat{A}(u_N, \bm{x}, s)\|^2\d s \\
    &\hspace{4cm} + \|u(\bm{x}, 0) - u_N(\bm{x}, 0)\|^2\bigg).
    \end{aligned}
\end{equation}
\end{theorem}

\begin{proof}
    We denote $f(t)\coloneqq\|u(\bm{x}, t) - u_N(\bm{x}, t)\|^2$. We have:
    \begin{equation}
    \begin{aligned}
        \partial_t f(t) &= 2(A(u, \bm{x},t) - A(u_N, \bm{x},t), u-u_N) + 2(A(u_N, \bm{x},t) - \hat{A}(u_N, \bm{x},t), u - u_N)\\
        &\quad\leq 2L(u-u_N, u-u_N) +  2\|A(u_N, \bm{x}, t)-\hat{A}(u_N, \bm{x}, t)\|\cdot\|u(\bm{x}, t) - u_N(\bm{x}, t)\|\\
        &\quad\quad\leq (2L+1)\|u-u_N\|^2 + \|A(u_N, \bm{x}, t)-\hat{A}(u_N, \bm{x}, t)\|^2
        \end{aligned}
        \label{model}
    \end{equation}
    In Eq.~\eqref{model}, $(u, v)\coloneqq \int_\Omega u\cdot v\d\bm{x}$. Applying the Gronwall's inequality to Eq.~\eqref{model}, we have:
    \begin{equation}
        \begin{aligned}
        \|u(\bm{x}, t)-u_N(\bm{x}, t)\|^2\leq \exp\big((2L+1)t\big)\cdot \bigg(2 \int_0^t 
         \|A(u_N, \bm{x}, s)-\hat{A}(u_N, \bm{x}, s)\|^2\d\bm{s} \\+ \|u(\bm{x}, 0) - u_N(\bm{x}, 0)\|^2\bigg),
        \end{aligned}
        \label{temporal_error_bound}
    \end{equation}
    which proves Theorem~\ref{theorem2}.
\end{proof}

From inequality~\eqref{temporal_error_bound}, minimizing the right-hand side is a necessary condition for ensuring that the $L^2$ error between $u(\bm{x},t)$ and its approximation $u_N(\bm{x},t)$ remains small. If $A\!\left(u_N(\bm{x},t), \bm{x}, t\right) \in F_k^d$ for all $t$, then even when the scales and centers of the RBFs are time-invariant, i.e.,
\[
\bm{\epsilon}_j(t) \equiv \bm{\epsilon}=(\epsilon_1,\ldots,\epsilon_d)
\quad \text{and} \quad
\bm{x}_j(t) \equiv \bm{x}_j(0),
\]
Theorem~\ref{theorem1} implies that
\begin{equation}
\begin{aligned}
        &\min_{\hat{A}, u_N(\bm{x}, 0)} \bigg(2 \int_0^T  \|A(u_N, \bm{x}, t)-\hat{A}(u_N, \bm{x}, t)\|^2\d t + \|u(\bm{x}, 0) - u_N(\bm{x}, 0)\|^2\bigg) \\
 & \leq 2 \bigg(\sum_{i=1}^d\epsilon_i^{\frac{1}{2}}\cdot \big(\int_0^T\|\partial_{x_i}A(u_N, \bm x, t)\|_{\infty, 0}\d t + \|\partial_{x_i}u_0(\bm{x})\|_{\infty, 0}\big) \\
 &\quad+ 2\big(\int_0^T\|A(u_N, \bm x, t)\|_{\infty, 0}\text{d}t + \|u_0(\bm{x})\|\big)\cdot\big(1-\Phi(\bm\epsilon^{-\frac{1}{2}})\big)\bigg) \\
    &\quad\quad+ \prod_{i=1}^d\epsilon_i^{-k}2^{d+2k}c_{d, k}N^{-\frac{k}{2}}(\log N)^{(k+2)(d-1)+1}\|B_{\bm{1}, \bm{0}}\|_{\infty, k},
\end{aligned}
\label{time_error}
\end{equation}
suggesting that if $A(u_N, \bm{x}, t)$ and $u_N(\bm{x}, 0)$ are also spatially anisotropic, the ``curse of dimensionality" might be partially alleviated.

 On the other hand, Eq.~\eqref{time_error} only provides an upper bound for the error bound 
 \begin{equation}
     \bigg(2 \int_0^T  \|A(u_N, \bm{x}, t)-\hat{A}(u_N, \bm{x}, s)\|^2\d t + \|u(\bm{x}, 0) - u_N(\bm{x}, 0)\|^2\bigg).
 \end{equation}
 Suppose the initial condition $u(\bm{x}, 0)$ can be approximated well using $u_N(\bm{x}, 0)$ by minimizing $\|u(\bm{x}, 0)-u_N(\bm{x}, 0)\|^2$. Then, 
 we can consider adjusting scales and centers $\bm\epsilon_j(t)$ and $\bm x_j(t)$ over time for each RBF so that the first term in the error bound Eq.~\eqref{time_error}:
 \begin{equation}
     \int_0^T  \|A(u_N, \bm{x}, t)-\hat{A}(u_N, \bm{x}, t)\|^2\d t
     \label{minimize_quantity}
 \end{equation}
 can be minimized, leading to a small approximation error $\|u_N(\bm{x}, t)-u(\bm{x}, t)\|^2$. Therefore, our adaptive RBF strategy is to minimize Eq.~\eqref{minimize_quantity} with time-varying coefficients $c_i(t)$, scales $\bm \epsilon_i(t)$, and centers $\bm x_i(t)$ of the RBFs so that the upper bound for the error $\|u(\bm{x}, t)-u_N(\bm{x}, t)\|^2$ can be small. We adopt a neural network (NN), which takes the RBF approximation $u_N(\bm{x}, t)$ and time $t$ as the input and then outputs $\hat{A}$, to approximate the integrodifferential operator $A(u_N, \bm{x}, s)$ in Eq.~\eqref{model_problem}. The structure of the NN we use is described in Fig.~\ref{fig:snn}.
    \begin{figure}
    \centering
\includegraphics[width=0.9\linewidth]{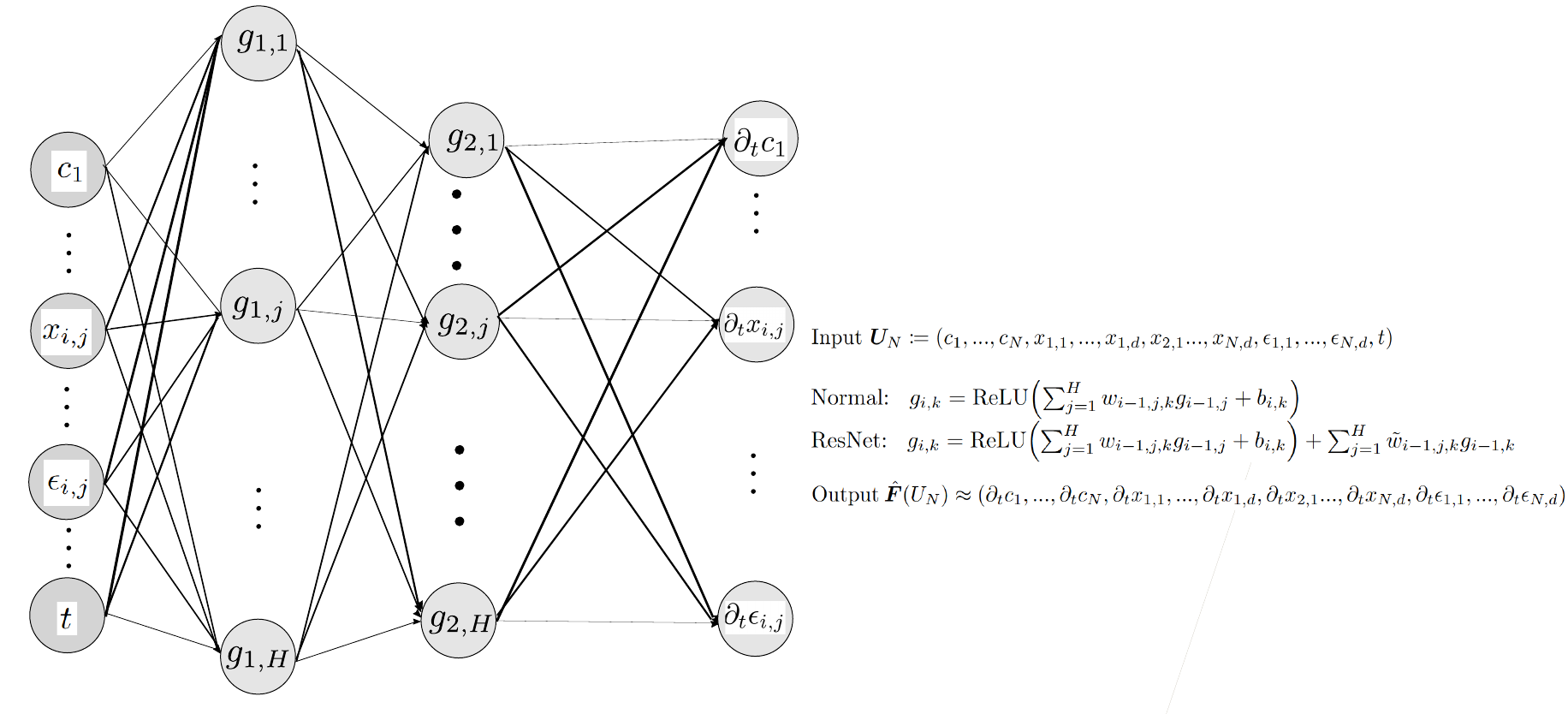}
    \caption{The structure of the NN, as the neural ODE model, used in this work. In the NN model, the coefficients $c_i(t)$, centers $\bm{x}_i(t)$, and scales of the RBFs $\bm{\epsilon}_i(t)$ in Eq.~\eqref{spatiotemporal_approx} as well as time $t$ are inputs of the NN. The output is a vector standing for the temporal derivatives $\partial_t c_i(t)$, $\partial_t \bm{x}_i(t)=(\partial_t x_{i, 1}(t),...,\partial_t x_{i, d}(t))$, and $\partial_t\bm{\epsilon}_i(t)=(\partial_t \epsilon_{i, 1}(t),...,\partial_t \epsilon_{i, d}(t))$ for $i=1,...,N$.
    Either the normal feedforward structure or the ResNet \cite{he2016deep} structure is used for forward propagation.}
    \label{fig:snn}
\end{figure}

\textbf{Remark:} the condition Eq.~\eqref{A_condition} can be met by various types of integrodifferential operators $A$. For example, if $A$ is the Laplace operator, then with the homogeneous Dirichlet boundary condition in $u$:
\begin{equation}
\begin{aligned}
    \int_{\Omega} u(\bm{x}, t)A(u(\bm{x}, t), \bm{x}, t)\d\bm{x} &= \int_{\Omega} -\nabla u(\bm{x}, t)\cdot \nabla u(\bm{x}, t)\d\bm{x} + \int_{\partial\Omega}u(\bm{x}, t)\frac{\partial u(\bm{x}, t)}{\partial \nu}\d\bm{x}\\
    &\quad =\int_{\Omega} -\nabla u(\bm{x}, t)\cdot \nabla u(\bm{x}, t)\d\bm{x}\leq 0,
    \end{aligned}
\end{equation}
which satisfies Eq.~\eqref{A_condition}. For simplicity, we assume that $A$ in Eq.~\eqref{model_problem} is linear.
Extensions to more complicated nonlinear scenarios, such as analysis on the cases in which the operator $A$ in Eq.~\eqref{model_problem} is nonlinear \cite{deng2025adaptive}, are left for future work.

\section{Numerical experiments}
\label{section3}
In this section, we present a series of numerical experiments to evaluate the effectiveness of the proposed adaptive RBF method for solving multidimensional spatiotemporal integrodifferential equations, for which it is complicated and computationally expensive to apply traditional mesh-based approaches. All numerical experiments are implemented in Python~3.11 and performed on a desktop workstation equipped with a 32-core Intel\textsuperscript{\textregistered} i9-13900KF CPU. Given the initial condition, we obtain an initial RBF approximation by minimizing the quantity $\|u_N(\bm{x},0) - u(\bm{x},0)\|$.

In the following examples, to solve the multidimensional spatiotemporal integrodifferential equation~\eqref{model_problem} subject to non-homogeneous Dirichlet boundary conditions, we train the NN shown in Fig.~\ref{fig:snn} by minimizing the following loss function:
\begin{equation}
\begin{aligned}
    \text{Loss} &= \sum_{j=1}^{\frac{T}{\Delta t}}\bigg(\frac{1}{N_0}\sum_{i=1}^{N_0} (A(u_N, \bm x_i, t_j) - \hat{A}(u_N, \bm x_i, t_j))^2\\
    &\hspace{2cm}+\frac{1}{N_1}\sum_{i=1}^{N_1} \big(u(\bm y_i, t_j) - u_N(\bm y_i, t_j))^2\bigg)\Delta t,
    \end{aligned}
    \label{loss_rbf}
\end{equation}
where $\bm{x}_i \in \Omega$ and $\bm{y}_i \in \partial \Omega$ are randomly resampled at each training epoch. The operators $A$ and $\hat{A}$ denote the ground-truth and approximate integrodifferential operators in Eqs.~\eqref{model_problem} and \eqref{approximate_model}, respectively, and $\Delta t$ is the time step. A pseudocode description of the proposed algorithm is provided in Algorithm~\ref{algorithm_1}. The settings and hyperparameters used in all numerical examples are summarized in Table~\ref{tab:setting} in Appendix~\ref{training_details}. For all experiments, we employ the \texttt{torchdiffeq} package in \texttt{PyTorch} \cite{chen2018neural} to numerically solve the ODEs governing the coefficients $c_i(t)$, scales $\bm{\epsilon}_i(t)$, and centers $\bm{x}_i(t)$ in the RBF approximation model~\eqref{approximate_model}.

\begin{algorithm}
\footnotesize
\caption{\footnotesize Pseudocode of the adaptive RBF approach for numerically solving multidimensional spatiotemporal equations.}
\begin{algorithmic}
  \STATE \textbf{Input:} initial condition $u(\bm{x},0)$, time step $\Delta t$, number of time steps $M$, and the number of training epochs $\text{epoch}_{\max}$.
  \STATE Initialize the neural ODE model shown in Fig.~\ref{fig:snn}.
  \STATE Construct an initial RBF approximation $u_N(\bm{x},0)$ of $u(\bm{x},0)$, and record the corresponding coefficients $c_i(0)$, centers $\bm{x}_i(0)$, and scales $\bm{\epsilon}_i(0)$.
  \FOR{$j = 1,\ldots,\text{epoch}_{\max}$}
    \STATE Use the \texttt{odeint} function with the approximate operator $\hat{A}(u_N,\bm{x},t)$ to compute $u_N(\bm{x},t_j)$ for $j = 1,\ldots,M$.
    \STATE Randomly sample $N_0$ interior points $\{\bm{x}_i\}_{i=1}^{N_0}$ and $N_1$ boundary points $\{\bm{y}_i\}_{i=1}^{N_1}$ for evaluating the loss function in Eq.~\eqref{loss_rbf}.
    \STATE Evaluate the loss function defined in Eq.~\eqref{loss_rbf}.
    \STATE Update the parameters of the neural ODE model by minimizing the loss function using gradient descent.
  \ENDFOR
  \STATE \textbf{Output:} the trained neural ODE model in Eq.~\eqref{approximate_model} and the numerical solutions $\{u_N(\bm{x},t_j)\}_{j=1}^{M}$.
\end{algorithmic}
\label{algorithm_1}
\end{algorithm}

A key advantage of the adaptive RBF approach over existing adaptive spectral methods is that the spatial domain $\Omega$ in Eq.~\eqref{spatiotemporal_approx} is not restricted to box-shaped domains or the unbounded space $\mathbb{R}^d$, where spectral basis functions can be constructed explicitly. Consequently, the proposed adaptive RBF method offers significantly greater flexibility for problems posed on arbitrary spatial domains. To illustrate this capability, we first consider a multidimensional spatiotemporal equation defined on an irregular domain.

\begin{example}
\label{example1}
    \rm
We numerically solve the following convection-diffusion equation: 
\begin{equation}
\begin{aligned}
    &\partial_tu(\bm{x}, t) = \frac{1}{2}\sum_{i=1}^d \partial_{x_ix_i}^2u - \sum_{i=1}^d b_i \partial_{x_i}u \\
    &\quad+ \sum_{i=1}^d \big(\frac{x_i-b_it}{t+a_i} + b_i\big)\exp\Big(- \frac{(x_i - b_it)^2}{2(t+a_i)}\Big) \cos(x_i) \\
    &\quad\quad+ \sum_{i=1}^d(\frac{1}{2(t+a_i)}+\frac{1}{2})\exp\Big(- \frac{(x_i - b_it)^2}{ 2(t+a_i)}\Big) \sin(x_i),\,\, \bm{x}\in\Omega, t\in[0, T],
    \end{aligned}
    \label{example1_model}
\end{equation}
with the initial condition:
\begin{equation}
    u(\bm{x}, 0) = \sum_{i=1}^d\exp\big(- \frac{x_i^2}{ 2a_i}\big)\cdot\sin(x_i)
    \label{example1_ic}
\end{equation}
and the Dirichlet boundary condition:
\begin{equation}
    u(\bm{x}, t) = \sum_{i=1}^d\exp\big(- \frac{(x_i - b_it)^2}{ 2(t+a_i)}\big)\cdot\sin(x_i), \,\, \bm{x}\in\partial\Omega.
    \label{example1_bc}
\end{equation}
The spatial domain $\Omega \coloneqq \Big\{(x_1,...,x_d):\sum_{i=1}^dx_i=0, x_i\in\big[-\frac{3}{d}, 3 - \frac{3}{d}\big]\Big\}$.  
Eq.~\eqref{example1_model}, together with the initial condition Eq.~\eqref{example1_ic} and Eq.~\eqref{example1_bc}, admits an analytical solution:
\begin{equation}
    u(\bm{x}, t) = \sum_{i=1}^d\exp\big(- \frac{(x_i - b_it)^2}{ 2(t+a_i)}\big)\cdot\sin(x_i).
    \label{example1_analytic}
\end{equation}
We set $a_i = 0.5 + 0.5i$ and $b_i=0.05i$ in Eqs.~\eqref{example1_model},\eqref{example1_ic}, \eqref{example1_bc}, and \eqref{example1_analytic}.
When sampling points on the boundary $\partial\Omega$ for evaluating the loss function Eq.~\eqref{loss_rbf}, for each boundary point, we use the following strategy: i) sample a surface out of the $d+1$ surfaces of $\Omega$ with equal probability  ($\frac{1}{d+1}$), and ii) after randomly selecting the surface, each point on the surface has equal probability density with respect to surface area. The error used in this example refers to the squared $L^2$ error to measure the accuracy of the numerical solution:
\begin{equation}
    \text{Error at time~} t\coloneqq \frac{\|u_N(\bm{x}, t) - u(\bm{x}, t)\|^2}{\|u(\bm{x}, t)\|^2}.
\end{equation}

    \begin{figure}
    \centering
\includegraphics[width=0.9\linewidth]{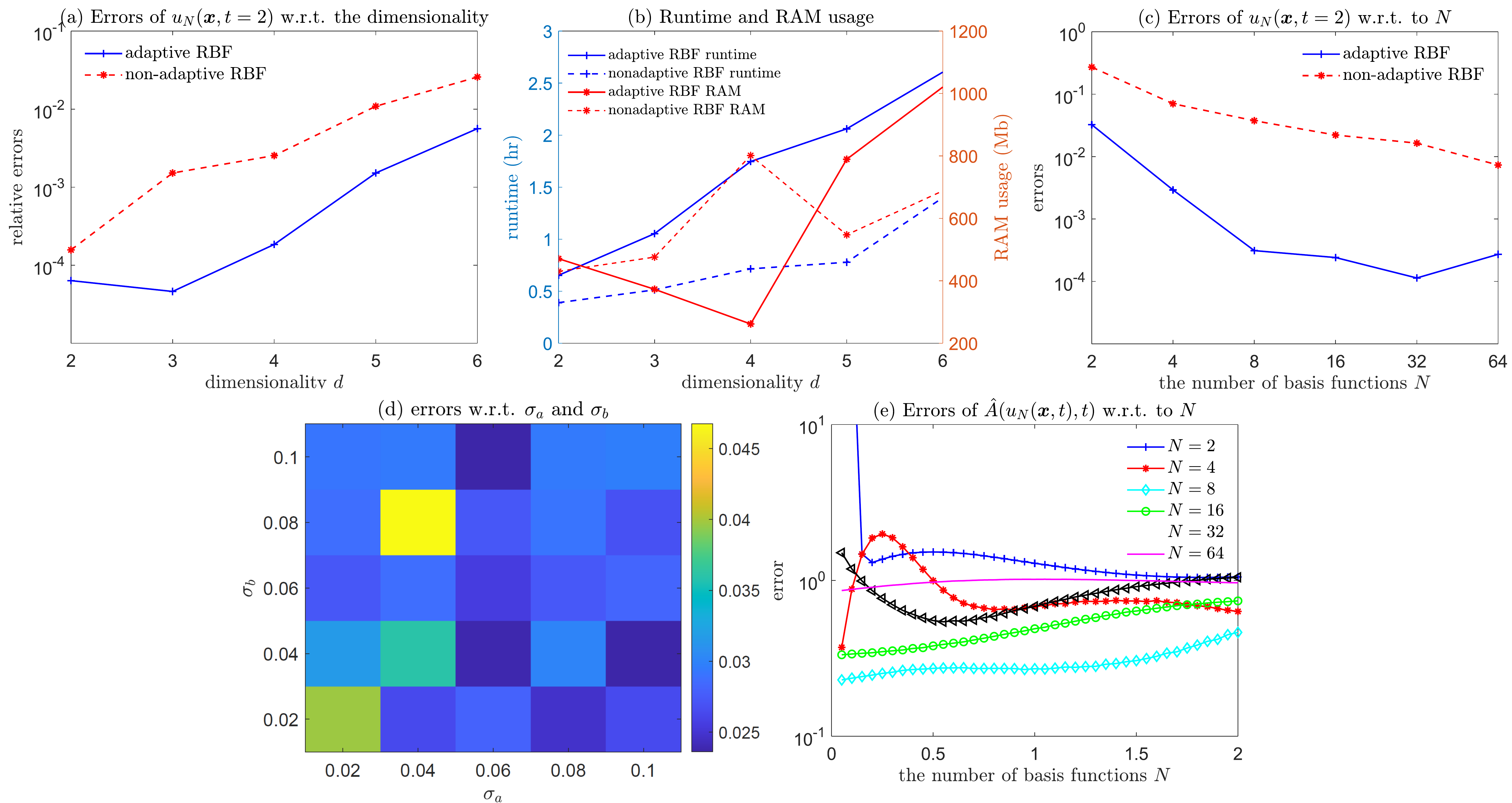}
    \caption{(a)(b) errors and computational costs of adaptive RBF approach versus non-adaptive RBF approach (with fixed centers and shapes over time) with different dimensionality $d$. The number of basis functions $N=30$. (c) errors and computational costs of the adaptive RBF approach versus the non-adaptive RBF approach with different numbers of basis functions, with the dimensionality $d=3$ and the number of basis functions $N=32$. (d) errors w.r.t. the noise $\sigma_a, \sigma_b$ in the model parameters $a_i, b_i, i=1,2,3$. $N=32$ and $d=3$. (e) errors when solving the inverse problem of inferring Eq.~\eqref{inverse_learn}, which is the RHS of Eq.~\eqref{example1_model} from observed spatiotemporal data. The dimensionality $d=3$, and the coefficients $a_i = 0.5 + 0.5i$ and $b_i=0.05i$. The error refers to the relative error $\frac{\int_0^T\|A(u_N(\bm{x}, t), t)-\hat{A}(u_N(\bm{x}, t), t)\|^2\d t}{\int_0^T\|A(u_N(\bm{x}, t), t)\|^2\d t}$, where $A$ refers to the RHS of Eq.~\eqref{example1_model}.}
    \label{fig:example1}
\end{figure}

Specifically, we compare with a non-adaptive RBF approach, \textit{i.e.}, using an RBF approximation in which the centers and scales of the basis functions do not change over time. From Fig.~\ref{fig:example1} (a), adaptively adjusting the scales and centers of the RBFs over time greatly improves the accuracy of numerically solving Eq.~\eqref{example1_model}, compared to using RBFs with fixed centers and scales. 
Yet, evolving the centers and scales of the RBFs also leads to a moderate increase in the computational time and RAM usage (shown in Fig.~\ref{fig:example1} (b). Although the error of our adaptive RBF approach increases with dimensionality $d$, this increase is moderate, and the magnitude of the error remains less than 2 when the dimensionality increases from 2 to 6, indicating that our adaptive RBF has the potential to efficiently solve multidimensional spatiotemporal equations. When the number of basis functions is increased from 2 to 32, the error of our adaptive RBF approach is decreased, as shown in Fig.~\ref{fig:example1} (c). Furthermore, given the same number of basis functions, our adaptive RBF always outperforms the non-adaptive RBF counterpart. However, when the number of basis functions is increased from 32 to 64, the error does not decrease, suggesting that some basis functions are redundant or not allocated optimally.

Our adaptive RBF approach can also be extended to the simultaneous solution of a family of spatiotemporal equations. Rather than fixing the parameters $a_i$ and $b_i$, we generate 50 groups of independently sampled parameters
$a_i \sim \mathcal{N}(0.5 + 0.5 i, \sigma_a^2), \,\, 
b_i \sim \mathcal{N}(0.05 i, \sigma_b^2), \,\, i = 1, 2, 3$.
For each parameter set $(a_1,a_2,a_3,b_1,b_2,b_3)$, the corresponding initial condition~\eqref{example1_ic} and boundary condition~\eqref{example1_bc} are generated accordingly. In this setting, we simultaneously solve 50 instances of Eq.~\eqref{example1_model} with different realizations of the model parameters $(a_1,a_2,a_3,b_1,b_2,b_3)$. 
Consequently, in addition to the RBF coefficients $c_i$, centers $x_{i,j}$, and scales $\epsilon_{i,j}$, the parameters $(a_1,a_2,a_3,b_1,b_2,b_3)$ are also provided as inputs to the neural ODE model shown in Fig.~\ref{fig:snn}. The results, presented in Fig.~\ref{fig:example1}(d), demonstrate that the proposed adaptive RBF approach is largely insensitive to the magnitude of parameter uncertainties when applied to the simultaneous solution of a family of spatiotemporal equations.


On the other hand, we can consider the inverse problem of learning the RHS of Eq.~\eqref{example1_model}:
\begin{equation}
\begin{aligned}
&\frac{1}{2}\sum_{i=1}^d \partial_{x_ix_i}^2u - \sum_{i=1}^d b_i \partial_{x_i}u 
    + \sum_{i=1}^d \big(\frac{x_i-b_it}{t+a_i} + b_i\big)\exp\Big(- \frac{(x_i - b_it)^2}{2(t+a_i)}\Big) \cos(x_i) \\
 &\hspace{2cm}   + \sum_{i=1}^d\big(\frac{1}{2(t+a_i)}+\frac{1}{2}\big)\cdot\exp\Big(- \frac{(x_i - b_it)^2}{ 2(t+a_i)}\Big) \sin(x_i)
    \end{aligned}
    \label{inverse_learn}
\end{equation}
using the approximate $\hat{A}$ in Eq.~\eqref{approximate_model} from the observed spatiotemporal data $u(\bm{x}, t)$. When $u(\bm{x}, 0)$ can be well approximated by $u_N(\bm{x}, 0)$ and $u(\bm{x}, t)$ can be well approximated by $u_N(\bm{x}, t)$ for $\bm x\in\partial\Omega$, keeping a small $\int_0^T \|u(\bm x, t)-u_N(\bm x, t)\|^2\d t$ is a necessary condition such that the RHS of Eq.~\eqref{temporal_error_bound} is small from Theorem~\ref{theorem2}. 
Therefore, we minimize the following loss function to train the NN in Fig.~\ref{fig:snn} to approximate $\hat{A}$:
\begin{equation}
    \text{Loss} = \sum_{j=1}^{\frac{T}{\Delta t}}\sum_{i=1}^{N_0} \|u_N(\bm x_i, t_j) - u(\bm x_i, t_j)\|^2\Delta t.
    \label{loss_rbf_inverse}
\end{equation}

A pseudocode description of the proposed adaptive RBF approach for learning a general integrodifferential operator $A$ in Eq.~\eqref{model_problem} is provided in Appendix~\ref{pseudocode_inverse}, which closely parallels Algorithm~\ref{algorithm_1}. As shown in Fig.~\ref{fig:example1}(e), learning the spatial differential operator on the right-hand side of Eq.~\eqref{example1_model} from data is more challenging than directly solving Eq.~\eqref{example1_model}, as the error of the learned spatial differential operator remains on the order of $\mathcal{O}(10^{-1})$.

Finally, as an additional numerical study, we examine how the neural network architecture-the number of hidden layers, the number of neurons per layer, and the use of a ResNet feedforward structure affects the accuracy of solving Eq.~\eqref{example1_model}. The detailed results are presented in Appendix~\ref{example1_appendix}. Our results indicate that employing too few layers or too few neurons per layer leads to degraded accuracy, while incorporating the ResNet architecture improves performance as the network depth increases. Further investigation is warranted to determine how deeper neural networks with more sophisticated architectures can further enhance the solution of multidimensional spatiotemporal equations.


\end{example}
Next, we consider numerically solving a nonlinear spatiotemporal differential equation to further examine the effectiveness of our adaptive RBF approach.

\begin{example}
\rm
\label{example2}
We solve a nonlinear 2D Burgers' equation in \cite{gao2017analytical}:
\begin{equation}
\begin{aligned}
    &\frac{\partial u}{\partial t} + u \frac{\partial u}{\partial x_1} + v \frac{\partial u}{\partial x_2} - \mu \left( \frac{\partial^2 u}{\partial x^2_1} + \frac{\partial^2 u}{\partial x_2^2} \right) = 0, \quad\\
    &\frac{\partial v}{\partial t} + u \frac{\partial v}{\partial x_1} + v \frac{\partial v}{\partial x_2} - \mu \left( \frac{\partial^2 v}{\partial x^2_1} + \frac{\partial^2 v}{\partial x_2^2} \right) = 0, \,\, \bm x=(x_1, x_2)\in(0, 1), \,\, t\in[0, T], 
    \end{aligned}
    \label{burgers_equation}
\end{equation}
where $\mu^{-1}$ is the Reynolds number. The initial conditions and boundary conditions are:
\begin{equation}
\begin{aligned}
    &u(x, y, 0)  = \sin(\pi x) \cos(\pi y),\,\,
    v(x, y, 0)  = \cos(\pi x) \sin(\pi y),\\
    &u(0, y, t) = u(1, y, t) = v(x, 0, t) = v(x, 1, t) = 0,\\
    u(x_1, x_2, t) &= \frac{2\pi\mu \displaystyle\sum_{n=0}^{\infty} \sum_{m=0}^{\infty} n A_{mn} C_{nm} E_{mn}(t) \sin(n\pi x_1) (1-2x_2)^m}
{\displaystyle\sum_{n=0}^{\infty} \sum_{m=0}^{\infty} A_{mn} C_{nm} E_{mn}(t) \cos(n\pi x_1)(1-2x_2)^m}, \,\,x_2=0, 1 \\
v(x_1, x_2, t) &= \frac{2\pi\mu \displaystyle\sum_{n=0}^{\infty} \sum_{m=0}^{\infty} m A_{mn} C_{nm} E_{mn}(t) (1-2x_1)^m \sin(m\pi x_2)}
{\displaystyle\sum_{n=0}^{\infty} \sum_{m=0}^{\infty} A_{mn} C_{nm} E_{mn}(t)(1-2x_1)^m  \cos(m\pi x_2)}, \,\,x_1=0, 1,
    \end{aligned}
    \label{ic_bc}
\end{equation}
where 
\begin{align}
&\hspace{-1cm}C_{mn} = \int_0^1 \int_0^1 \exp\left[2\lambda \cos(\pi x_1) \cos(\pi x_2)\right] \cos(n\pi x_1) \cos(m\pi x_2)  \d x_1  \d x_2,\,\, \lambda = \frac{1}{4\mu\pi},\\
E_{mn}(t) &= \exp\left[ -\left( n^2 + m^2 \right) \pi^2 \mu t \right], \\
A_{mn} &= 
\begin{cases}
1, & \text{if } n = 0 \text{ and } m = 0 \\
2, & \text{if } n = 0 \text{ and } m \neq 0 \\
2, & \text{if } n \neq 0 \text{ and } m = 0 \\
4, & \text{if } n \neq 0 \text{ and } m \neq 0
\end{cases}. 
\end{align}

Under the initial conditions and boundary conditions Eq.~\eqref{ic_bc}, the 2D Burgers' equation admits an analytical solution:
\begin{equation}
\begin{aligned}
u(x_1, x_2, t) &= \frac{2\pi\mu \displaystyle\sum_{n=0}^{\infty} \sum_{m=0}^{\infty} n A_{mn} C_{nm} E_{mn}(t) \sin(n\pi x_1) \cos(m\pi x_2)}
{\displaystyle\sum_{n=0}^{\infty} \sum_{m=0}^{\infty} A_{mn} C_{nm} E_{mn}(t) \cos(n\pi x_1) \cos(m\pi x_2)}, \\
v(x_1, x_2, t) &= \frac{2\pi\mu \displaystyle\sum_{n=0}^{\infty} \sum_{m=0}^{\infty} m A_{mn} C_{nm} E_{mn}(t) \cos(n\pi x_1) \sin(m\pi x_2)}
{\displaystyle\sum_{n=0}^{\infty} \sum_{m=0}^{\infty} A_{mn} C_{nm} E_{mn}(t) \cos(n\pi x_1) \cos(m\pi x_2)}.
\end{aligned}
\label{example2_analytic}
\end{equation}

We set $\mu=0.1$ in Eq.~\eqref{burgers_equation}. From Eq.~\eqref{example2_analytic}, the analytical solution to Eq.~\eqref{burgers_equation} converges to 0 as $t$ increases for all $\bm{x}$. Therefore, we use the time-averaged relative errors:
\begin{equation}
    \text{Error}\coloneqq \frac{\int_0^T \|u(\bm x, t) - u_N(\bm x, t)\|^2\d t }{\int_0^T \|u(\bm x, t)\|^2\d t}.
    \label{time_average_error}
\end{equation}

    \begin{figure}
    \centering
\includegraphics[width=0.9\linewidth]{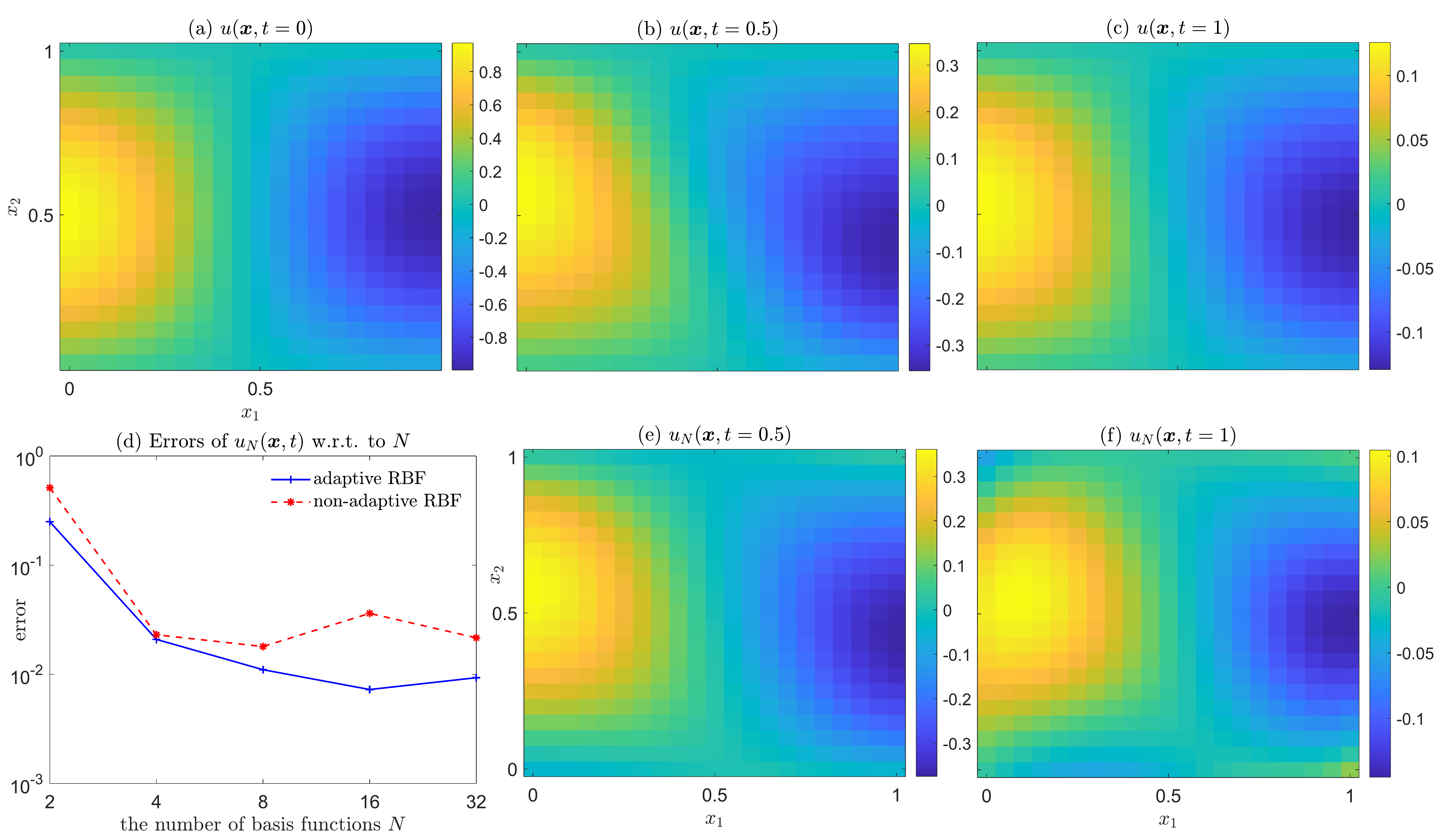}
    \caption{(a)(b)(c) The analytical solution $u(\bm{x}, t)=u(x_1, x_2, t)$ (Eq.~\eqref{example2_analytic}) at $t=0, 0.5, 1$. (d) The time-averaged errors Eq.~\eqref{time_average_error} of the adaptive and non-adaptive RBF approaches with different numbers of basis functions. (e)(f) The adaptive RBF approximation $u_N(\bm{x}, t)=u(x_1, x_2, t)$ at $t=0.5, 1$ with $N=16$.}
    \label{fig:example2}
\end{figure}

We present the initial condition $u(\bm{x},0)$ with $\bm{x}=(x_1,x_2)$, together with the ground-truth solutions $u(\bm{x},0.25)$ and $u(\bm{x},1)$ in Fig.~\ref{fig:example2}(a)–(c). We also display the adaptive RBF approximations $u_N(\bm{x},t)$ at $t=0.25$ and $t=1$ in Fig.~\ref{fig:example2}(e)–(f), which are in close agreement with the corresponding analytical solutions. Furthermore, as shown in Fig.~\ref{fig:example2}(d), increasing the number of basis functions from 2 to 16 leads to a notable improvement in the accuracy of the numerical solution to Eq.~\eqref{burgers_equation}. However, the error obtained with 32 basis functions is comparable to that with 16, indicating that additional basis functions may be redundant. Moreover, Fig.~\ref{fig:example2}(d) demonstrates that a reduction in approximation error with increasing basis functions is achieved only when adaptive techniques are employed. For the non-adaptive RBF method, the approximation error stagnates once the number of basis functions exceeds 4, suggesting that fixed RBFs are unable to effectively adjust their locations over time. Overall, these results highlight the importance of efficiently allocating basis functions in the spatial domain as their number increases, particularly when solving multidimensional spatiotemporal equations.
\end{example}

Finally, we apply the proposed adaptive RBF approach to numerically solve a multidimensional, nonlinear spatiotemporal integrodifferential equation posed on an unbounded domain. Unlike traditional finite difference or finite element methods, which rely on spatial discretization and require elaborate domain truncation techniques when treating problems on unbounded domains, the RBFs we use are inherently defined over the entire $\mathbb{R}^d$. Consequently, the proposed adaptive RBF method can be directly applied to spatiotemporal equations whose spatial variables are defined in unbounded domains. 
\begin{example}
\rm
    \label{example3}
Consider a 4D Vlasov equation characterizing swarming systems of interacting and self-propelled discrete particles in \cite{carrillo2009double,muntean2014collective}:
\begin{equation}
    \begin{aligned}
    &\frac{\partial u}{\partial t}(\bm{x}, \bm{v}, t) + v \cdot \nabla_{\bm x} u + \text{div}_{\bm v}\left[(\alpha - \beta|\bm v|^2)\bm v u\right] \\
    &\quad- (\int_{\mathbb{R}^2}\nabla_{\bm x} U(\bm{y}) \rho(\bm{x}-\bm{y})\d \bm{y})\cdot\nabla_{\bm v} u = 0,\,\,
    \bm{x}, \bm{v}\in\mathbb{R}^2,\,\, t\in[0, T].
    \end{aligned}
    \label{example3_equation}
\end{equation}
In Eq.~\eqref{example3_equation}, $u(\bm{x}, \bm{v}, t)$ denotes the particle density with respect to position $\bm{x} \coloneqq (x_1,x_2)$ and velocity $\bm{v} \coloneqq (v_1,v_2)$ at time $t$. The term $\nabla_{\bm{v}} u$ represents the two-dimensional gradient of $u$ with respect to the velocity variables. Moreover,
\[
\rho(\bm{x},t) \coloneqq \int_{\mathbb{R}^2} u(\bm{x},\bm{v},t)\,\mathrm{d}\bm{v}
\]
denotes the spatial density of particles. We assume that $u(\bm{x},\bm{v},t) \to 0$ as $\|\bm{x}\|,\|\bm{v}\| \to \infty$. The function $U$ denotes the Morse potential, which consists of attractive and repulsive components:
\begin{equation}
    U(r) = -C_a e^{-\frac{r}{\ell_a}} + C_r e^{-\frac{r}{\ell_r}},
\end{equation}
where $C_a$ and $C_r$ denote the strengths of the attractive and repulsive interactions, respectively, and $\ell_a$ and $\ell_r$ are the corresponding length scales. We set $C_a = 8$, $C_r = 24$, $\ell_a = 3$, $\ell_r = 0.1$, $\alpha = 0.05$, $\beta = 0.1$. 
It was found in \cite{carrillo2009double} that, under mild initial conditions, particles exhibit clockwise or counterclockwise rotation around a central point after a transient period. Motivated by this behavior, we apply the proposed adaptive RBF approach to solve Eq.~\eqref{example3_equation} with the following initial condition:
\begin{equation}
    u(\bm{x}, \bm{v}, 0)=\frac{4}{\pi^2}\prod_{i=1}^2\exp(-2x_i^2)\cdot \prod_{i=1}^2\exp(-2v_i^2).
\end{equation}

Since both the spatial variable $\bm{x}$ and the velocity variable $\bm{v}$ are defined on the unbounded domain $\mathbb{R}^2$, we sample $(\bm{x},\bm{v}) \sim \mathcal{N}(\bm{0}, (\frac{5}{4})^2\cdot I_4)$. Moreover, based on the assumption that the analytical solution satisfies $u(\bm{x},\bm{v},t) \to 0$ as $\|\bm{x}\|,\|\bm{v}\| \to \infty$, we impose an additional constraint on the RBF approximation by enforcing a lower bound on the scales of each basis function, namely that each component of $\bm{\epsilon}$ satisfies $\bm{\epsilon} \ge \epsilon_0 = 0.1$. As a result, both the RBF approximation and the analytical solution vanish as $\bm{x}$ or $\bm{v}$ approaches infinity. With this vanishing boundary condition at infinity, we employ the following revised form of the loss function~\eqref{loss_rbf}:
 \begin{equation}
 \begin{aligned}
     &\hspace{-1cm}\text{Loss} = \sum_{j=1}^{\frac{T}{\Delta t}}\bigg(\sum_{i=1}^{N_0} (\partial_tu_N(\bm x_i, t_j) - \hat{A}(u_N(\bm x_i, t_j), \bm x_i, t_j))^2\\&\hspace{1cm}+ \lambda \big(\int_{\mathbb{R}^4} u_N(\bm{x}, \bm{v},t_0)\d \bm{x}\d\bm{v} - \int_{\mathbb{R}^4} u_N(\bm{x}, \bm{v},t_j)\d \bm{x}\d\bm{v}\big)^2\bigg)\Delta t.
     \end{aligned}
    \label{loss_rbf_example4}
 \end{equation}
 In Eq.~\eqref{loss_rbf_example4}, $\hat{A}$ denotes the learned spatial integrodifferential operator associated with the RBF approximation~\eqref{approximate_model}. The second term in the loss function serves as a regularization term that enforces conservation of the particle density, and we set $\lambda=2$. We plot the RBF approximation of the particle density $\int_{\mathbb{R}^2} u_N(\bm{x}, \bm{v}, t)\,\mathrm{d}\bm{v}$ for various values of $\bm{x}$, together with the corresponding scaled velocity field:
\begin{equation}
    \Big(\frac{\int_{\mathbb{R}^2}v_1u_N(\bm{x}, \bm{v}, t)\d\bm{v}}{\int_{\mathbb{R}^2}u_N(\bm{x}, \bm{v}, t)\d\bm{v}}, \frac{\int_{\mathbb{R}^2}v_2u_N(\bm{x}, \bm{v}, t)\d\bm{v}}{\int_{\mathbb{R}^2}u_N(\bm{x}, \bm{v}, t)\d\bm{v}}\Big).
\label{weighted_velocity}
\end{equation}

    \begin{figure}
    \centering
\includegraphics[width=0.9\linewidth]{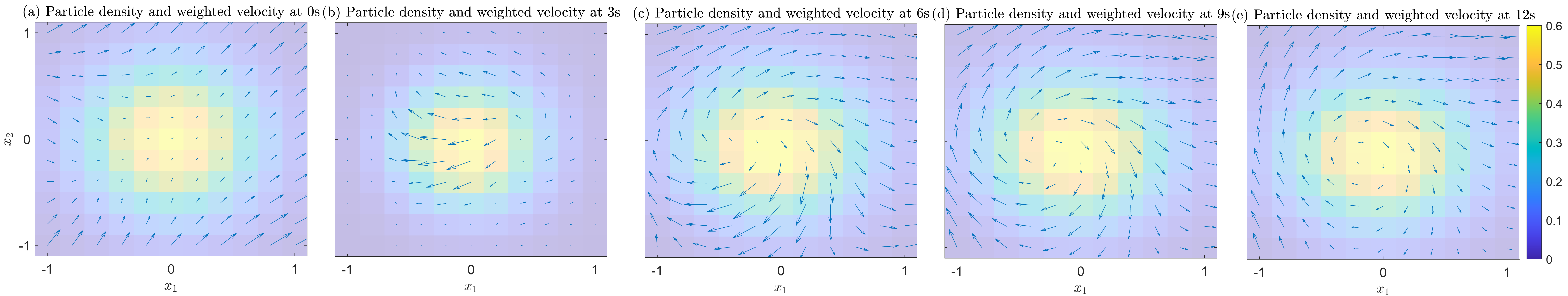}
    \caption{
    The particle density $\int_{\mathbb{R}^2} f(\bm{x}, \bm{v}, t)\,\mathrm{d}\bm{v}$ and the weighted velocity field~\eqref{weighted_velocity} at different times $t$, obtained using the proposed adaptive RBF method to numerically solve the 4D Vlasov equation~\eqref{example3_equation}.}
    \label{fig:example3}
\end{figure}

As shown in Fig.~\ref{fig:example3}, after a transient initial period, the particle density becomes concentrated within a circular region centered at the origin, and the weighted velocity field exhibits rotational motion around the origin. This behavior is consistent with the agent-based particle simulations reported in \cite{carrillo2009double} for Eq.~\eqref{example3_equation}. These results demonstrate that, even with only 40 basis functions in the RBF approximation, the proposed adaptive RBF approach can efficiently solve the four-dimensional spatiotemporal integrodifferential equation~\eqref{example3_equation}.

\end{example}

\section{Summary and conclusion}
\label{section4}
In this manuscript, we proposed an adaptive RBF approach for the efficient solution of multidimensional spatiotemporal integrodifferential equations. The proposed method dynamically and automatically adjusts the centers and scales of the RBFs over time, which is essential for accurately capturing the evolving dynamics of multidimensional spatiotemporal solutions. We further analyzed how RBF-based approximations can partially mitigate the curse of dimensionality when the solution exhibits spatial heterogeneity. From a numerical perspective, we demonstrated the efficacy of the proposed approach by solving a range of multidimensional spatiotemporal equations.

Several promising directions remain for future research. It would be valuable to investigate how different choices of radial basis functions-such as algebraically decaying RBFs or compactly supported basis functions \cite{wendland2017multiscale}-influence the accuracy and efficiency of spatiotemporal solvers. Additionally, developing efficient RBF-based frameworks for learning spatiotemporal integrodifferential equations from data represents a promising avenue for further study. Another important direction is the design of strategies for optimally allocating RBFs, including the use of multiscale radial basis functions \cite{billings2007generalized,kedward2017efficient} or adaptive control of the number of RBFs \cite{shao2025solving}. Also, it will be helpful to carry out a systematic comparison of our approach against traditional numerical methods, such as the finite element method, and machine-learning-based PDE and integrodifferential equation solvers on solving complex multidimensional spatiotemporal equations in terms of both computational costs and accuracy.
Finally, it would be of interest to   
explore whether incorporating deep neural networks with more advanced architectures could further enhance the capability of the proposed adaptive RBF approach for solving higher-dimensional spatiotemporal equations.

\section*{Acknowledgements}
The authors thank Prof. Robert Azencott and Prof. Ilya Timofeyev at the University of Houston and Prof. Philip K. Maini at the University of Oxford for their valuable comments and suggestions on this manuscript. 

\section*{Data Availability Statement}
No new data were created in this research. The code used in this research will be made publicly available upon acceptance of this manuscript.

%
 \section*{Conflict of interest}
 The authors declare that they have no conflict of interest.
 
\appendix
\section*{Appendix}
\section{Proof to Theorem~\ref{theorem1}}
\label{proof_theorem1}
In this section, we present the proof of Theorem~\ref{theorem1}. Since $u(\bm{x})$ vanishes on the boundary $\partial \Omega$, we define
\begin{equation}
    u_{\bm\epsilon}(\bm{x})\coloneqq \int_{\Omega}u(\bm{y})B_{\bm{\epsilon}^{-1}, \bm0}(\bm{x}-\bm{y})\d\bm{y}=\int_{\mathbb{R}^d}u(\bm{y})B_{\bm{\epsilon}^{-1}, \bm 0}(\bm{x}-\bm{y})\d\bm{y}.
\end{equation}
Then, we have:
\begin{equation}
\begin{aligned}
        |u_{\bm\epsilon}(\bm{x}) - u(\bm{x})| &\leq \int_{\mathbb{R}^d}u(\bm{y})B_{{\bm \epsilon}^{-1}, \bm0}(\bm{x}-\bm{y})\d\bm{y}-\int_{\mathbb{R}^d}u(\bm{x})B_{{\bm \epsilon}^{-1}, \bm0}(\bm{x}-\bm{y})\d\bm{y}\\
        &\leq \sum_{i=1}^d\|\partial_{x_i}u\|_{\infty, 0}\epsilon_i^{\frac{1}{2}} (1-\Phi(\epsilon^{-\frac{1}{2}})) + 2\big(1 - \Phi(\epsilon^{-\frac{1}{2}})\big)\|u\|_{\infty, 0}.
        \end{aligned}
        \label{rbf_error}
\end{equation} 

We consider the hyperbolic cross space constructed via Smolyak formulas based on polynomial interpolation at the extrema of Chebyshev polynomials, as proposed in \cite{barthelmann2000high}. The $j^{\text{th}}$ coordinate of the collocation points is given by:
\begin{equation}
    X_{i_j}\coloneqq \Big\{x_i^j=-\cos\left( \frac{\pi (j - 1)}{m_i - 1} \right), \quad j = 1, \ldots, m_i\Big\}.
\label{knots}
\end{equation}
In Eq.~\eqref{knots}, $m_i=1$ if $i=1$ and $m_i=2^{i-1}+1$ for $i>1$. The collocation points of the hyperbolic cross space are given as:
\begin{equation}
    H(q, d) = \bigcup_{q-d+1 \leq \|\mathbf{i}\| \leq q} (X_{i_1} \otimes \cdots \otimes X_{i_d}).
    \label{nest_grid}
\end{equation}
We denote by $n(q,d) \coloneqq |H(q,d)|$ the number of collocation points in the set $H(q,d)$. By the construction of the nested grids in Eq.~\eqref{nest_grid}, it follows from \cite{novak1999simple} that
\begin{equation}
2^{\,q-1} \;\le\; n(q,d) \;\le\; \binom{q-1}{d-1} \cdot2^{\,q}
\;\le\; q^{\,d-1} 2^{\,q}.
\end{equation}

The Smolyak algorithm for interpolation is given by
\begin{equation}
A(q, d) = \sum_{|\mathbf{i}| \leq q} (\Delta_{i_1} \otimes \cdots \otimes \Delta_{i_d})
\label{aqd}
\end{equation}
for integers $q > d$, where $\Delta_{i}\coloneqq \mathcal{I}_i-\mathcal{I}_{i-1}$, and $\mathcal{I}_i$ is the one-dimensional interpolation operator:
\begin{equation}
\mathcal{I}_i(f)(x_i^j) = f(x_i^j), \,\, x_i^h\in X_{i_j}.
\end{equation}

By the exactness of the Clenshaw--Curtis collocation points established in \cite{novak1996high}, the interpolation formula $A(q,d)$ defined in Eq.~\eqref{aqd} is exact on the following function space (i.e., the integral of any function in this space is uniquely determined by its values at the collocation points):
\begin{equation}
\mathbb{P}_{q,d} \coloneqq 
\sum_{\|\mathbf{i}\|_1 = q}
\left(
P_{x_1}(m_1-1) \otimes P_{x_2}(m_2-1) \otimes \cdots \otimes P_{x_d}(m_d-1)
\right),
\end{equation}
where $P_{x_i}(m_i-1)$ denotes the space of univariate polynomials in the variable $x_i$ of degree at most $m_i-1$.

We further define $\mathcal{I}_{q,d}$ as the $d$-dimensional interpolation operator associated with the Clenshaw--Curtis collocation points, such that for any $f \in \mathbb{P}_{q,d}$,
\begin{equation}
   f(\bm{x}) = \mathcal{I}_{q,d} f(\bm{x}), \qquad \forall\, \bm{x} \in H(q,d). 
\end{equation}

Then, there exist weights $\{w_i\}_{i=1}^N$ such that
\begin{equation}
\sum_{i=1}^N u(\bm{x}_i)\, w_i \, B_{\bm{\epsilon}^{-1},\bm{0}}(\bm{x}-\bm{x}_i)
=
\int_{\Omega}
\mathcal{I}_N^d\!\left(
u(\bm{y})\, B_{\bm{\epsilon}^{-1},\bm{0}}(\bm{x}-\bm{y})
\right)\, \mathrm{d}\bm{y},
\end{equation}
where $q(N,d) \in \mathbb{N}^+$ is a function of $N$ and $d$ such that the number of radial basis functions $N$ satisfies:
\begin{equation}
    n(q(N, d), d)\leq N < n(q(N, d)+1, d). 
\end{equation}
When $N$ is large enough such that $2^{\frac{q(N, d)-1}{d-1}}>q+1$, we conclude that 
\begin{equation}
    \frac{N^{\frac{1}{2}}}{2}\leq 2^{q-1}\leq  n(q, d).
\end{equation}
Observe that $\|u(\bm{y})B_{\bm\epsilon^{-1}}(\bm{x}-\bm{y})\|_{\infty, k}\leq 2^k\|u(\bm{y})\|_{\infty, k}\|B_{\bm\epsilon^{-1}}\|_{\infty, k}$. Using \cite[Theorem 8]{barthelmann2000high}, there exists a constant $c_{d, k}$ such that:
\begin{equation}
\begin{aligned}
       &|u_{\bm\epsilon}(\bm{x}) - \sum_{i=1}^N u(\bm{x}_i)w_iB_{\bm\epsilon^{-1},\bm0}(\bm{x}-\bm{x}_i)|\leq \int_{\Omega} u(\bm{y})B_{\bm\epsilon^{-1},\bm0}(\bm{x}-\bm{y})\d\bm{y} \\
       &\hspace{7cm}-  \int_{\Omega} \mathcal{I}_N^d (u(\bm{y}) B_{\bm\epsilon^{-1},\bm0}(\bm{x}-\bm{y}))\d\bm{y}\\
       &\leq 2^dc_{d, k}n(q(N, d), d)^{-k}\log n(q(N, d), d)^{(k+2)(d-1)+1}\|u(\bm{y})\|_{\infty, k}\|B_{\bm\epsilon^{-1},\bm0}(\bm{x}-\bm{y})\|_{\infty, k}\\
       &\leq \prod_{i=1}^d\epsilon_i^{-k}2^dc_{d, k}n(q(N, d), d)^{-k}\log n(q(N, d), d)^{(k+2)(d-1)+1}\|u(\bm{y})\|_{\infty, k}\|B_{\bm1,\bm0}(\bm{y})\|_{\infty, k}.
\end{aligned}
\end{equation}

Therefore, we conclude that:
\begin{equation}
    \begin{aligned}
        &\big|u(\bm{x}) - \sum_{i=1}^Nw_iu(\bm{x}_i)B_{\bm\epsilon^{-1}, \bm0}(\bm{x}-\bm{x}_i)\big|\leq |u(\bm{x}) - u_{\bm\epsilon}(\bm{x})| + \big|u_{\bm \epsilon}(\bm{x}) - \sum_{i=1}^Nw_iu(\bm{x}_i)B_{\bm \epsilon^{-1},\bm0}(\bm{x}-\bm{x}_i)\big|\\
        &\hspace{2cm}\leq \bigg(\sum_{i=1}^d\epsilon_i^{\frac{1}{2}}\|\partial_{x_i}u\|_{\infty, 0} + 2\|u\|_{\infty, 0}\big(1-\Phi(\bm\epsilon^{-\frac{1}{2}})\big)\bigg) \\
        &\hspace{3cm}+ \prod_{i=1}^d\epsilon_i^{-k}2^{d+2k}c_{d, k}N^{-\frac{k}{2}}(\log N)^{(k+2)(d-1)+1}\|u\|_{\infty, k}\|B_{\bm 1, \bm0}\|_{\infty, k},
    \end{aligned}
    \label{error_bound0}
\end{equation}
which proves Theorem~\ref{theorem1}. 

\section{Settings and hyperparameters of numerical experiments}
\label{training_details}
We list the hyperparameters and settings for each example in Table~\ref{tab:setting}.

{\scriptsize \begin{table}[h!]
\centering
\caption{\footnotesize Settings and hyperparameters for numerical experiments for each example. NN parameters include weights $w_{i, j, k}$, the weights $\tilde{w}_{i, j, k}$ for the ResNet technique, as well as biases $b_{i, k}$ in Fig.~\ref{fig:snn}.} 
{\scriptsize\begin{tabular}{lllll}
\toprule
 & Example~\ref{example1} & Example~\ref{example2}  & Example~\ref{example3}\\
\midrule
gradient descent method & Adam & Adam  & Adam\\
forward propagation method &  ResNet  & ResNet & ResNet\\
learning rate & 0.001  & 0.005  & 0.00002\\
number of epochs &2000  & 400  & 500 \\
Time horizon $T$ & 2  & 1 & 12  \\
Time step $\Delta t$ &0.05 & 0.05 & 0.15\\
Number of interior sample points $N_0$ & 300 & 300 & 300\\
Number of boundary sample points $N_1$ & 300 & 300 & $\backslash$\\
number of hidden layers in the NN & 3  &2 & 2 \\
activation function &  GELU  &GELU & GELU\\
number of neurons in each layer & 100  &300 & 200  \\
initialization for $w_{i, j, k}$ and $b_{i, j}$  & $\mathcal{N}(0, 0.001^2)$ & $\mathcal{N}(0, 0.001^2)$  & $\mathcal{N}(0, 0.005^2)$ & \\

\bottomrule
\end{tabular}}
\label{tab:setting}
\end{table}}

\section{Pseudocode of the adaptive RBF method for inverse problems in spatiotemporal integrodifferential equations}
\label{pseudocode_inverse}

\begin{algorithm}
\footnotesize
\caption{\footnotesize Pseudocode of the adaptive RBF approach for solving the inverse problem of inferring an integrodifferential operator in a multidimensional spatiotemporal equation from observed data.}
\begin{algorithmic}
  \STATE \textbf{Input:} initial condition $u(\bm{x},0)$, time step $\Delta t$, number of time steps $M$, and the number of training epochs $\text{epoch}_{\max}$.
  \STATE Initialize the neural ODE model shown in Fig.~\ref{fig:snn}.
  \STATE Construct an initial RBF approximation $u_N(\bm{x},0)$ of $u(\bm{x},0)$, and record the associated coefficients $c_i(0)$, centers $\bm{x}_i(0)$, and scales $\bm{\epsilon}_i(0)$.
  \FOR{$j = 1,\ldots,\text{epoch}_{\max}$}
    \STATE Use the \texttt{odeint} function with the approximate operator $\hat{A}(u_N,\bm{x},t)$ to compute $u_N(\bm{x},t_j)$ for $j = 1,\ldots,M$.
    \STATE Randomly sample $N_0$ interior points $\{\bm{x}_i\}_{i=1}^{N_0}$ and $N_1$ boundary points $\{\bm{y}_i\}_{i=1}^{N_1}$ for evaluating the loss function in Eq.~\eqref{loss_rbf_inverse}.
    \STATE Evaluate the loss function defined in Eq.~\eqref{loss_rbf_inverse}.
    \STATE Update the parameters of the neural ODE model by minimizing the loss function using gradient descent.
  \ENDFOR
  \STATE \textbf{Output:} the trained neural ODE model. The learned integrodifferential operator is then obtained via Eq.~\eqref{approximate_model}.
\end{algorithmic}
\label{algorithm_2}
\end{algorithm}

\section{Analysis on the width and depth of the neural network}
\label{example1_appendix}
We empirically investigate the impact of neural network width and depth on the accuracy of the proposed adaptive neural-network-based RBF approach. Specifically, we solve Eq.~\eqref{example1_model} from Example~\ref{example1} using neural networks with different architectures to represent the approximate operator $\hat{A}$ in Eq.~\eqref{approximate_model}, where the spatial dimension is $d=3$. The corresponding results are summarized in Table~\ref{tab:nn_structure}.
\begin{table}[h!]
\scriptsize
\centering
  \caption{Errors in the adaptive RBF approximation $u_N(\bm{x}, t=2)$ for Example~\ref{example1}. Here, ResNet indicates the use of the ResNet technique described in Fig.~\ref{fig:snn}. All hyperparameters and settings are identical to those of Example~\ref{example1}, as listed in Table~\ref{tab:setting}. The adaptive RBF approximation employs 30 basis functions.}
\begin{tabular}{ccccc}
\toprule  width & \# of layers & ResNet & initialization for weights \& biases & error   \\ 
\midrule 
25 & 3 & Yes & $\mathcal{N}(0, 0.001^2)$& $1.64\times10^{-4}$ \\ 
50 & 3 & Yes & $\mathcal{N}(0, 0.001^2)$& $3.14\times10^{-5}$  \\ 
 100 & 3 & Yes & $\mathcal{N}(0, 0.001^2)$ & $9.71\times10^{-5}$  \\
  200 & 3 & Yes & $\mathcal{N}(0, 0.001^2)$ & $1.06\times10^{-4}$  \\
    100 & 1 & Yes & $\mathcal{N}(0, 0.001^2)$&  $2.50\times10^{-3}$
 \\ 
  100 & 2 & Yes & $\mathcal{N}(0, 0.001^2)$& $9.22\times10^{-5}$
  \\ 
100 & 4 & Yes & $\mathcal{N}(0, 0.001^2)$& $7.57\times10^{-5}$  \\ 
100 & 5 & Yes & $\mathcal{N}(0, 0.001^2)$& $5.48\times10^{-5}$  \\ 
   100 & 1 & No & $\mathcal{N}(0, 0.001^2)$& $1.80\times10^{-4}$  \\ 
100 & 2 & No & $\mathcal{N}(0, 0.001^2)$& $2.21\times10^{-5}$  \\ 
100 & 3 & No & $\mathcal{N}(0, 0.001^2)$& $1.39\times10^{-4}$ \\ 
100 & 4 & No & $\mathcal{N}(0, 0.001^2)$& $5.66\times10^{-4}$ \\ 
100 & 5 & No & $\mathcal{N}(0, 0.001^2)$& $4.41\times10^{-4}$ \\ 
\bottomrule
\end{tabular}
\label{tab:nn_structure}
\end{table}

\newpage

\bibliographystyle{spmpsci}      


\end{document}